\documentclass[runningheads,envcountsame,a4paper]{llncs} 

\usepackage{amssymb,amsmath}
\usepackage{wasysym}
\usepackage{graphicx}
\usepackage{epsfig}
\usepackage{latexsym}
\usepackage[english]{babel}
\usepackage[utf8]{inputenc}
\usepackage{csquotes}
\usepackage{algorithm} 
\usepackage{algpseudocode}
\usepackage{tikz}
\usetikzlibrary{arrows,shapes.geometric,positioning,calc}
\usepackage[colorinlistoftodos]{todonotes}
\usepackage{changepage}
\usepackage{subfig}
\usepackage{soul}
\usepackage{hyperref}
\usepackage{todonotes}
\usepackage{enumerate}

\newtheorem{thm}{Theorem}
\newtheorem{prop}[thm]{Proposition}
\newtheorem{cor}[thm]{Corollary}
\newtheorem{lem}[thm]{Lemma}

\newtheorem{exa}[thm]{Example}

\newsavebox{\qedB}

\sbox{\qedB}{\setlength{\unitlength}{1mm}
 \begin{picture}(4,4)(0,0)
  \thinlines
  {\put(0,0){\framebox(2.83,2.83){}}}%
  {\put(1.17,1.17){\framebox(2.83,2.83){}}}%
  {\put(0,0){\framebox(4,4){}}}%
  {\put(1.17,1.17){{\rule{1ex}{1ex} }}}%
 \end{picture}}
 
\newcommand{\QEDB}{\ifmmode\def\next{\tag"\usebox{\qedB}"}%
 \else\let\next=\relax
 {\unskip\nobreak\hfil\penalty50
 \hskip2em\hbox{}\nobreak\hfil\usebox{\qedB}
 \parfillskip=0pt \finalhyphendemerits=0\penalty-100\bigskip}\fi\next}

\newcommand{\Alphabet}{\hbox{\rm Alph}}

\newcommand{\fac}{\hbox{\rm Fac}}

\newcommand{\prim}{\hbox{\rm Prim}}

\newcommand{\Sc}{\hbox{\rm sc}}

\newcommand{\bprop}{\begin{prop}}
\newcommand{\eprop}{\end{prop}}
\newcommand{\bcor}{\begin{cor}}
\newcommand{\ecor}{\end{cor}}
\newcommand{\blem}{\begin{lem}}
\newcommand{\elem}{\end{lem}}

\title{An upper bound of the number of distinct powers in binary words}

\author{ Shuo Li}

\institute{Laboratoire de Combinatoire et d'Informatique Mathématique,\\
Université du  Québec \`a Montréal,\\
CP 8888 Succ. Centre-ville, Montréal (QC) Canada H3C 3P8\\
\email{li.shuo@uqam.ca} 
}

\begin{document}

\maketitle


\begin{abstract}

A \emph{power} is a word of the form $\underbrace{uu...u}_{k \; \text{times}}$, where $u$ is a word and $k$ is a positive integer and a \emph{square} is a word of the form $uu$. Fraenkel and Simpson conjectured in 1998 that the number of distinct squares in a word is bounded by the length of the word. This conjecture was proven recently by Brlek and Li. Besides, there exists a stronger upper bound for binary words conjectured by Jonoska, Manea and Seki stating that for a word of length $n$ over the alphabet $\left\{a, b\right\}$, if we let $k$ be the least of the number of a's and
the number of b's and $k \geq 2$, then the number of distinct squares is
upper bounded by $\frac{2k-1}{2k+2}n$. In this article, we prove this conjecture by giving a stronger statement on the number of distinct powers in a binary word.
  
\end{abstract}

\section{Introduction}
\label{sec:intro}
A \emph{power} is a word of the form $\underbrace{uu...u}_{k \; \text{times}}$, where $u$ is a word and $k$ is a positive integer; the power is also called a {\em $k$-power} and $k$ is its {\em exponent}. The upper bound of the number of distinct $k$-powers in a finite word was studied in~\cite{KUBICA,li2022,li202202} and the best known result is as follows:
\begin{thm}[Theorem 1 and Theorem 2 in~\cite{li202202} ]
\label{th:sw}
 For every finite word $w$, let $m(w)$ denote the number of distinct nonempty powers of exponent at least 2 in $w$, let $m_k(w)$ denote the number of distinct nonempty $k$-powers in $w$, let $|w|$ denote the length of $w$ and let $|\Alphabet(w)|$ denote the number of distinct letters in $w$, then one has $$m(w) \leq |w| - |\Alphabet(w)|;$$ Moreover, for any integer $k \geq 2$, $$m_k(w) \leq \frac{|w| - |\Alphabet(w)|}{k-1}.$$
\end{thm}
Particularly, a square is a $2$-power and upper bound of the number of distinct squares in a finite word was studied in ~\cite{FraenkelS98,Ilie,lam,DezaFT15,thie,Brlekli}. A conjecture of Fraenkel and Simpson~\cite{FraenkelS98} states that the number of distinct squares in a word is bounded by the length of the word. This conjecture is confirmed by the previous theorem.
Besides, there exists a stronger upper bound for binary words conjectured by Jonoska, Manea and Seki~\cite{Jono}:

\begin{thm}
\label{th:conj}
Let $w$ be a finite word over the alphabet $\left\{a, b\right\}$, let $s(w)$ denote the number of distinct squares in $w$ and let $k$ denote the least of the number of a's and
the number of b's in $w$. If $k \geq 2$, then one has
 $$s(w) \leq \frac{2k-1}{2k+2}|w|.
 $$
\end{thm}

In this article, we confirm this conjecture by proving the following result:

\begin{thm}
\label{th:main}
Let $w$ be a finite word over the alphabet $\left\{a, b\right\}$ such that $w=a^{r_1}ba^{r_2}ba^{r_3}b...a^{r_k}ba^{r_{k+1}}$, where $r_1,r_2,...,r_{k+1}$ are nonnegative integers and $a^n=\underbrace{aa...a}_{n \; \text{times}}$. Let $\sigma$ be a permutation of $1,2,...,k+1$ such that $r_{\sigma(1)} \leq r_{\sigma(2)}\leq...\leq r_{\sigma(k+1)}$, let $m(w)$ denote the number of distinct nonempty powers of exponent at least 2 in $w$ and let $|w|$ denote the length of $w$. Then one has $$m(w)+r_{\sigma(k)} \leq |w| -2 .$$
\end{thm}

\section{Preliminaries}
\label{sec:prel}

Let $\sum$ be an \emph{alphabet} and $\sum^*$ be the set of all words over $\sum$. Let $w \in \sum^*$. By $|w|$, we denote its \emph{length}. A word of length $0$ is called the  \emph{empty word} and it is denoted by $\varepsilon$. A word $u$ is a \emph{factor} of $w$ if $w = pus$ for some words $p,s$. When $p = \varepsilon$ (resp. $s = \varepsilon$), $u$ is called a \emph{prefix} (resp. \emph{suffix}) of $w$. The set of all nonempty factors of $w$ is denoted by $\fac(w)$.
The number of occurrences of a factor $u \in w$ is denoted by $|w|_u$.

Let $w$ be a finite word. For any integer $i$ satisfying $1 \leq i \leq |w|$, let $L_w(i)$ be the set of all length-$i$ factors of $w$ and let $C_w(i)$ be the cardinality of $L_w(i)$.
For any natural number $k$, we define the {\em $k$-th power} of a finite word $u$ to be $u^k = u u \cdots u$ and it consists of the concatenation of $k$ copies of $u$. A finite word $w$ is said to be {\em primitive} if it is not a power of another word, that is, $w=u^k$ implies $k=1$. A \emph{square} 
is a $2$-power, that is a word $w$ satisfying $w = uu$ for a certain word $u$. 

For a finite word $w$, let $\prim(w)$ denote the set of primitive factors of $w$, let
$$M(w)=\left\{p^i|p^i \in \fac(w), p \in \prim(w), i \in \mathbf{N},i \geq 2\right\},$$
$$S(w)=\left\{p^{2i}|p^{2i} \in \fac(w), p \in \prim(w), i \in \mathbf{N},i \geq 1\right\},$$
$$NS(w)=\left\{p^{2i+1}|p^{2i+1} \in \fac(w), p \in \prim(w), i \in \mathbf{N},i \geq 1\right\},$$
and let $m(w)$,$s(w)$ and $ns(w)$ be respectively the cardinality of $M(w)$, $S(w)$ and $NS(w)$. Obviously, $M(w)=S(w) \cup NS(w)$, $m(w)=s(w)+ns(w)$ and $s(w)$ is the number of distinct nonempty squares in $w$.\\

Here we recall some elementary definitions and proprieties concerning graphs from Berge~\cite{berge}.

A {\em directed graph} consists of a nonempty set of {\em vertices} $V$ and a set of {\em edges} $E$. A vertex $a$ represents an endpoint of an edge and an edge joins two vertices $a, b$ in order.
A {\em chain} is a sequence of edges $e_1,e_2, \cdots, e_k$, such that there exists a sequence of vertices $v_1,v_2, \cdots, v_{k+1}$ and that for each $i$ satisfying $1 \leq i \leq k$, $e_i$ is either directed from $v_i$ to $v_{i+1}$ or from $v_{i+1}$ to $v_i$. A {\em cycle} is a finite chain such that $v_{k+1}=v_1$.
A {\em path} is a sequence of edges $e_1,e_2, \cdots, e_k$, such that there exists a sequence of vertices $v_1,v_2, \cdots, v_{k+1}$ and that for each $i$ satisfying $1 \leq i \leq k$, $e_i$ is directed from $v_i$ to $v_{i+1}$. A {\em circuit} is a finite path such that $v_{k+1}=v_1$.

A cycle or a circuit is called {\em elementary} if, apart from $v_1$ and $v_{k+1}$, every vertex which it meets is distinct. 
A directed graph is called {\em weakly connected} if for any couple of vertices $a,b$ in this graph, there exists a chain connecting $a$ and $b$.

Let $G$ be a weakly connected graph and let $\left\{e_1,e_2 \cdots e_l\right\}$, $\left\{v_1,v_2 \cdots v_s\right\}$ denote respectively the edge set and the vertex set of $G$. The number $\chi(G)=l-s+1$ is called the {\em cyclomatic number} of $G$.

Let $C$ be a cycle in $G$. A vector $\mu(C)=(c_1,c_2 \cdots c_l)$ in the $l$-dimensional space $\mathbb{R}^l$ is called the {\em vector-cycle corresponding to $C$} if $c_i$ is the number of visits of the edge $e_i$ in the cycle $C$ for all $i$ satisfying $1 \leq i \leq l$. The cycle $C_1,C_2, \cdots, C_k,...$ are said to be {\em independent} if their corresponding vectors are linearly independent. 

\begin{thm}[Theorem 2, Chapter 4 in~\cite{berge}]
\label{book}
the cyclomatic number of a graph is the maximum number of independent cycles in this graph.
\end{thm}
\section{Rauzy graphs}

Let $w$ be a finite word. For any integer $i$ satisfying $1\leq i\leq |w|$, the $i$-th Rauzy graph $\Gamma_w(i)$ of $w$ is defined to be an directed graph whose vertex set is $L_w(i)$ and the edge set is $L_w(i+1)$;
an edge $e \in L_w(i+1)$ starts at the vertex $u$ and ends at the vertex $v$, if $u$ is a
prefix and $v$ is a suffix of $e$. Let us define $\Gamma_w=\cup_{n=1}^{k-1}\Gamma_w(n)$.

Let $\Gamma_w(i)$ be a Rauzy graph of $w$ for some $i$, a sub-graph on $\Gamma_w(i)$ is called a {\em small circuit} if it is an elementary circuit and the number of its vertices is no larger than $i$. 

\begin{lem}[Lemma 8 in~\cite{Brlekli}]
\label{lem:inde}
Let $w$ be a finite word and let $\Gamma_w(i)$ be a Rauzy graph of $w$ for some
$i$. Then all small circuits on $\Gamma_w(i)$ are independent.
\end{lem}

\begin{lem}[Lemma 6 in~\cite{li202202}]
\label{lem:inj}
Let $w$ be a finite word, then there exists an injection from $M(w)$ to the set of all small circuits on $\Gamma_w$.
\end{lem}

\begin{exa}
Let us consider the word $u=abaaabaaaabaaba$, the Rauzy graph $\Gamma_u(4)$ is as follows:
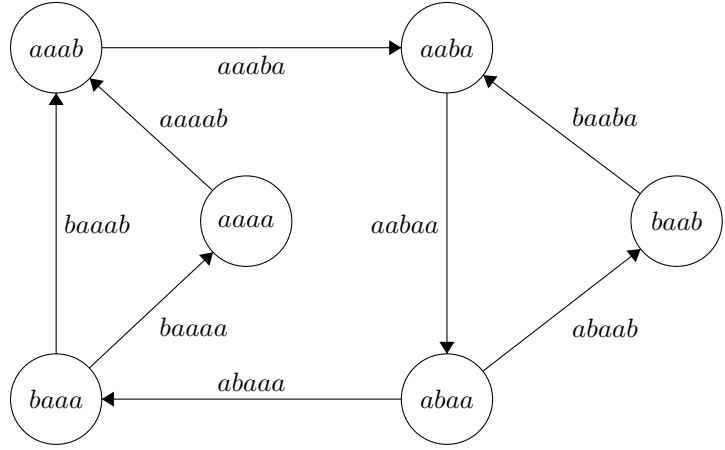
\begin{center}
\begin{tikzpicture}[scale=0.2]
\tikzstyle{every node}+=[inner sep=0pt]
\draw [black] (19.8,-19.4) circle (3);
\draw (19.8,-19.4) node {$aaab$};
\draw [black] (45.5,-19.4) circle (3);
\draw (45.5,-19.4) node {$aaba$};
\draw [black] (19.8,-42.7) circle (3);
\draw (19.8,-42.7) node {$baaa$};
\draw [black] (45.5,-42.7) circle (3);
\draw (45.5,-42.7) node {$abaa$};
\draw [black] (60.6,-30.9) circle (3);
\draw (60.6,-30.9) node {$baab$};
\draw [black] (32.3,-30.9) circle (3);
\draw (32.3,-30.9) node {$aaaa$};
\draw [black] (30.09,-28.87) -- (22.01,-21.43);
\fill [black] (22.01,-21.43) -- (22.26,-22.34) -- (22.94,-21.6);
\draw (28.84,-24.66) node [above] {$aaaab$};
\draw [black] (42.5,-42.7) -- (22.8,-42.7);
\fill [black] (22.8,-42.7) -- (23.6,-43.2) -- (23.6,-42.2);
\draw (32.65,-42.2) node [above] {$abaaa$};
\draw [black] (19.8,-39.7) -- (19.8,-22.4);
\fill [black] (19.8,-22.4) -- (19.3,-23.2) -- (20.3,-23.2);
\draw (20.3,-31.05) node [right] {$baaab$};
\draw [black] (22.8,-19.4) -- (42.5,-19.4);
\fill [black] (42.5,-19.4) -- (41.7,-18.9) -- (41.7,-19.9);
\draw (32.65,-19.9) node [below] {$aaaba$};
\draw [black] (45.5,-22.4) -- (45.5,-39.7);
\fill [black] (45.5,-39.7) -- (46,-38.9) -- (45,-38.9);
\draw (45,-31.05) node [left] {$aabaa$};
\draw [black] (21.98,-40.64) -- (30.12,-32.96);
\fill [black] (30.12,-32.96) -- (29.19,-33.14) -- (29.88,-33.87);
\draw (28.84,-37.28) node [below] {$baaaa$};
\draw [black] (47.86,-40.85) -- (58.24,-32.75);
\fill [black] (58.24,-32.75) -- (57.3,-32.85) -- (57.91,-33.63);
\draw (55.89,-37.3) node [below] {$abaab$};
\draw [black] (58.21,-29.08) -- (47.89,-21.22);
\fill [black] (47.89,-21.22) -- (48.22,-22.1) -- (48.83,-21.3);
\draw (55.89,-24.65) node [above] {$baaba$};
\end{tikzpicture}
\end{center}

In this graph, there are three circuits: $C_1=\left\{\left\{aaab,aaba,abaa,baaa\right\},\left\{aaaba,aabaa,abaaa,baaab\right\}\right\}$, $C_2=\left\{\left\{aaba,abaa,baab\right\},\left\{aabaa,abaab,baaba\right\}\right\}$ and \\$C_3=\left\{\left\{aaab,aaba,abaa,baaa,aaaa\right\},\left\{aaaba,aabaa,abaaa,baaaa,aaaab\right\}\right\}$. Two of them are small, they are $C_1$ and $C_2$, while $C_3$ is not small. \qed
\end {exa}

\section{Proof of Theorem~\ref{th:main} }
\begin{lem}
\label{lem:sp}
Let $w \in \left\{a,b\right\}^*$ and let $i$ be an integer satisfying $1 \leq i \leq |w|$. If there exists an elementary circuit on $\Gamma_w(i)$ containing the edge $a^ib$, then this circuit is not a small circuit and it is independent with all the small circuits on $\Gamma_w(i)$. Let $C_{sp}(i)$ denote one of these circuits (if any). 
\end{lem}

\begin{proof}
We first prove that, if there exists a small circuit passing through the vertex $a^i$, then it should be the sub-graph $ \left\{\left\{a^i\right\},\left\{a^{i+1}\right\}\right\}$ of $\Gamma_w(i)$. In fact, if there exists a path $e_1,e_2,...,e_k$ from $a^i$ to $a^i$ satisfying $k \leq i$, let $p=l_1l_2...l_k$ be a word such that $l_j$ is the last letter of $e_j$ for all $j$ satisfying that $1 \leq j \leq k$. From the hypothesis that $e_1,e_2,...,e_k$ form a circuit, we can deduce that $a^i$ is a suffix of $a^ip$. Moreover, as $|p|=k \leq i$, $l_j=a$ for all $j$. Consequently, $e_j=a^{i+1}$ for all $j$. Thus, from the unicity of each edge, we prove that there exists only one edge on the path and the graph is given by  $\left\{\left\{a^i\right\},\left\{a^{i+1}\right\}\right\}$.

If their exists an elementary circuit $C$ on $\Gamma_w(i)$ containing the edge $a^ib$, from the fact that $a^ib$ cannot be contained in any small circuit on $\Gamma_w(i)$, we conclude that $C$ is not a small circuit and independent with all the small circuits on $\Gamma_w(i)$.\qed

\end{proof}

\begin{lem}
\label{lem:circuit}
Let $w \in \left\{a,b\right\}^*$ such that $w=a^{r_1}ba^{r_2}ba^{r_3}b...a^{r_k}ba^{r_{k+1}}$ with $r_j \geq 0$ for all $j$ satisfying $1 \leq j \leq k+1$. Let $\sigma$ be a permutation of $1,2,...,k+1$ such that $r_{\sigma(1)} \leq r_{\sigma(2)}\leq...\leq r_{\sigma(k+1)}$. Then for any integer $i$ satisfying $1\leq i \leq \sigma(k)$, there exists a $C_{sp}(i)$ on $\Gamma_w(i)$.
\end{lem}

\begin{proof}
From the hypothesis that $i \leq r_{\sigma(k)} \leq r_{\sigma(k+1)}$, there exists a nonempty word $X$ such that $a^iXa^i \in \fac(w)$ and that 
$$X=\begin{cases}
bYb\; \text{if $|X| \geq 2$};\\
b \;\;\;\;\;\text{otherwise}.
\end{cases}
$$
Moreover, we can suppose that $a^i \not \in \fac(Y)$. Indeed, there exists a circuit on $\Gamma_{a^iXa^i}(i)$ containing the edge $a^ib$ and $\Gamma_{a^iXa^i}(i)$ is a sub-graph of $\Gamma_{w}(i)$. Thus, there exists a $C_{sp}(i)$ on $\Gamma_w(i)$.\qed
\end{proof}

\begin{lem}
\label{lem:number}
Let $w \in \left\{a,b\right\}^*$ and let $I_w$ be the cardinality of $$S_w=\left\{i| 1 \leq i \leq |w|, \; \text{there exists a  circuit $C_{sp}(i)$ on $\Gamma_w(i)$}\right\},$$ then $$m(w)+I_w \leq |w|- |\Alphabet(w)|.$$
\end{lem}
\begin{proof}
Let $\Sc_w(i)$ denote the number of small circuits on $\Gamma_w(i)$. From Lemma~\ref{lem:inde}, Lemma~\ref{lem:sp} and Theorem~\ref{book}, for any $i \in S_w$,
$$\Sc_w(i)+1 \leq C_w(i+1)-C_w(i)+1;$$
and for any $i \not \in S_w$,
$$\Sc_w(i) \leq C_w(i+1)-C_w(i)+1.$$ 
Consequently, $$\sum_{i=1}^{|w|}\Sc_w(i) + I_w \leq \sum_{i=1}^{|w|} C_w(i+1)-C_w(i)+1=|w|-|\Alphabet(w)|.$$
Moreover, from~\ref{lem:inj}, $m(w) \leq \sum_{i=1}^{|w|}\Sc(i)$, thus, $m(w)+I_w \leq |w|-|\Alphabet(w)|$.
\qed
\end{proof}

\begin{proof}[of Theorem~\ref{th:main}]
It is a direct consequence of Lemma~\ref{lem:number} and Lemma~\ref{lem:circuit}.
\qed
\end{proof}

\section{Proof of Theorem~\ref{th:conj}}

In this section, Let $w \in \left\{a,b\right\}^*$ such that $|w|_a \geq |w|_b$, that $|w|_b=k$ and that $w=a^{r_1}ba^{r_2}ba^{r_3}b...a^{r_k}ba^{r_{k+1}}$ with $r_j \geq 0$ for all $j$ satisfying $1 \leq j \leq k+1$.

Let $\delta$ be an integer such that $|w|_a=k+\delta$. From the hypothesis $|w|_a\geq |w|_b$, we can suppose $\delta \geq 0$ and $|w|=2k+ \delta$. Let $\delta=n(k+1)+i$ with $n \geq 0$ and $0 \leq i \leq k$. 

Let $\sigma$ be a permutation of $1,2,...,k+1$ such that $r_{\sigma(1)} \leq r_{\sigma(2)}\leq...\leq r_{\sigma(k+1)}$.

For any real number $x$, let $\lfloor x \rfloor$ to be the integer part of $x$.

\begin{lem}[Proposition 2 in~\cite{Jono} ]
\label{lem:pre}
Theorem~\ref{th:conj} holds if $k \leq 9$.
\end{lem}

\begin{lem}
\label{lem:imply1}
$$s(w)+ 1+\lfloor\frac{3n}{2}+\frac{3i}{2(k+1)}\rfloor \leq|w|-2 \implies s(w) \leq \frac{2k-1}{2k+2}|w|. $$
\end{lem}

\begin{proof}
\begin{align*}
&s(w) \leq \frac{2k-1}{2k+2}|w|\\
\Longleftarrow&s(w) \leq \lfloor\frac{2k-1}{2k+2}|w|\rfloor \\
\Longleftarrow&s(w) \leq|w|- \lfloor\frac{3}{2k+2}(2k+ \delta)\rfloor\\
\Longleftarrow&s(w) \leq|w|- \lfloor\frac{3k}{k+1}+\frac{3\delta}{2(k+1)}\rfloor\\
\Longleftarrow&s(w) \leq|w|-2- \lfloor\frac{k-2}{k+1}+\frac{3\delta}{2(k+1)}\rfloor\\
\Longleftarrow&s(w)+\lfloor\frac{k-2}{k+1}+\frac{3\delta}{2(k+1)}\rfloor \leq|w|-2\\
\Longleftarrow&s(w)+ 1+\lfloor\frac{3\delta}{2(k+1)}\rfloor \leq|w|-2.
\end{align*}

From the definition, $\delta=n(k+1)+i$, thus,
$$s(w)+ 1+\lfloor\frac{3n}{2}+\frac{3i}{2(k+1)}\rfloor \leq|w|-2 \implies s(w) \leq \frac{2k-1}{2k+2}|w|. $$ \qed
\end{proof}

\begin{lem}
\label{lem:imply2}
\begin{equation}
\label{eq:1}
 s(w)+1+\lfloor\frac{3n}{2}+\frac{3i}{2(k+1)}\rfloor \leq m(w)+ r_{\sigma(k)} \implies s(w) \leq \frac{2k-1}{2k+2}|w|.
 \end{equation}
Moreover,
\begin{equation}
\label{eq:2}
1+\lfloor\frac{3(n+1)}{2}\rfloor \leq \lfloor\frac{r_{\sigma(k+1)}-1}{2}\rfloor+ r_{\sigma(k)} \implies s(w) \leq \frac{2k-1}{2k+2}|w|.
 \end{equation}
\end{lem}

\begin{proof}
The first part is a direct consequence of Theorem~\ref{th:main} and Lemma~\ref{lem:imply1}.\\  For the second part, from the hypothesis that $i \leq k$, we have $$\lfloor\frac{3n}{2}+\frac{3i}{2(k+1)}\rfloor \leq \lfloor\frac{3(n+1)}{2}\rfloor.$$
Moreover, from the fact that $$\left\{a^{2i+1}| 1 \leq i \leq  \lfloor\frac{r_{\sigma(k+1)}-1}{2}\rfloor \right\} \subset NS(w),$$ we have $$s(w)+\lfloor\frac{r_{\sigma(k+1)}-1}{2}\rfloor \leq m(s).$$
 \qed
\end{proof}

\begin{lem}
\label{lem:case1}
If $r_{\sigma(k)} \geq n+2$, then

$$1+\lfloor\frac{3(n+1)}{2}\rfloor \leq r_{\sigma(k)}+ \lfloor\frac{r_{\sigma(k+1)}-1}{2}\rfloor. $$

\end{lem}

\begin{proof}
If $r_{\sigma(k)} \geq n+2$, then $r_{\sigma(k+1)} \geq n+2$.

 If $n$ is odd, let $n=2t+1$, then
\begin{align*}
r_{\sigma(k)}+ \lfloor\frac{r_{\sigma(k+1)}-1}{2}\rfloor&\geq (2t+1)+2+  \lfloor\frac{(2t+1)+2-1}{2}\rfloor\\
&\geq 3t+4\\
&\geq \lfloor\frac{3((2t+1)+1)}{2}\rfloor+1.
\end{align*}

If $n$ is even, let $n=2t$, then
\begin{align*}
r_{\sigma(k)}+ \lfloor\frac{r_{\sigma(k+1)}-1}{2}\rfloor&\geq (2t)+2+  \lfloor\frac{(2t)+2-1}{2}\rfloor\\
&\geq 3t+2\\
&\geq \lfloor\frac{3((2t)+1)}{2}\rfloor+1.
\end{align*}
 \qed
\end{proof}

\begin{lem}
\label{lem:case2}
If $r_{\sigma(k)} \leq n$ and $k \geq 8$ then

$$1+\lfloor\frac{3(n+1)}{2}\rfloor \leq r_{\sigma(k)}+ \lfloor\frac{r_{\sigma(k+1)}-1}{2}\rfloor. $$

\end{lem}

\begin{proof}
If $r_{\sigma(k)} \leq n$, let us suppose $r_{\sigma(k)}=m$. Then 
$$\sum_{j=1}^k r_{\sigma(j)} \leq mk.$$
From the fact that $\sum_{j=1}^{k+1} r_{\sigma(j)}=k+\delta,$
$$r_{\sigma(k+1)} \geq k+ \delta -mk \geq n(k+1)-mk+k.$$
Thus,
\begin{align*}
&1+\lfloor\frac{3(n+1)}{2}\rfloor \leq r_{\sigma(k)}+ \lfloor\frac{r_{\sigma(k+1)}-1}{2}\rfloor\\
\Longleftarrow&1+\frac{3(n+1)}{2} \leq r_{\sigma(k)}+\frac{r_{\sigma(k+1)}-1}{2}-1\\
\Longleftarrow&\frac{5}{2}+\frac{3n}{2} \leq m+\frac{n(k+1)-mk+k-1}{2}-1\\
\Longleftarrow&3n+5 \leq 2m+n(k+1)-mk+k-1-2\\
\Longleftarrow&3n+5 \leq n(k+1)-m(k-2)+k-3\\
\Longleftarrow&3n+5 \leq n(k+1)-n(k-2)+k-3\;\; (*)\\
\Longleftarrow&3n+5 \leq 3n+k-3\\
\Longleftarrow&8 \leq k.
\end{align*}
The relation (*) holds because $k-2 \geq 0$ and $m \leq n$ from hypothesis.
 \qed
\end{proof}

\begin{lem}
\label{lem:case3}
If $r_{\sigma(k)}=n+1$ and $r_{\sigma(k+1)}\geq n+4$, then

$$1+\lfloor\frac{3(n+1)}{2}\rfloor \leq r_{\sigma(k)}+ \lfloor\frac{r_{\sigma(k+1)}-1}{2}\rfloor. $$

\end{lem}

\begin{proof}
If $n$ is odd, let $n=2t+1$, then
\begin{align*}
r_{\sigma(k)}+ \lfloor\frac{r_{\sigma(k+1)}-1}{2}\rfloor&\geq (2t+1)+1+  \lfloor\frac{(2t+1)+4-1}{2}\rfloor\\
&\geq 3t+4\\
&\geq \lfloor\frac{3((2t+1)+1)}{2}\rfloor+1.
\end{align*}

If $n$ is even, let $n=2t$, then
\begin{align*}
r_{\sigma(k)}+ \lfloor\frac{r_{\sigma(k+1)}-1}{2}\rfloor&\geq (2t)+1+  \lfloor\frac{(2t)+4-1}{2}\rfloor\\
&\geq 3t+2\\
&\geq \lfloor\frac{3((2t)+1)}{2}\rfloor+1.
\end{align*}
 \qed
\end{proof}

\begin{lem}
\label{lem:case4}
If $r_{\sigma(k)}=n+1,r_{\sigma(k+1)}\leq n+3$ and $ k \geq 10$ then

$$s(w)+1+\lfloor\frac{3n}{2}+\frac{3i}{2(k+1)}\rfloor \leq s(w)+ r_{\sigma(k)}+ \lfloor\frac{r_{\sigma(k+1)}-1}{2}\rfloor \leq m(w)+ r_{\sigma(k)}. $$
\end{lem}

\begin{proof}
From the fact that $$\sum_{j=1}^{k+1}r_{\sigma(j)}=|w|_a=k+\delta=k+n(k+1)+i,$$ we have $$k+n(k+1)+i \leq (\sum_{j=1}^{k}n+1)+n+3.$$
Consequently, $i \leq 3$ and $\frac{3i}{2(k+1)}<\frac{1}{2}$. Thus, $\lfloor\frac{3n}{2}+\frac{3i}{2(k+1)}\rfloor=\lfloor\frac{3n}{2}\rfloor$.\\
We only need to prove $$\lfloor\frac{3n}{2}\rfloor+1 \leq r_{\sigma(k)}+ \lfloor\frac{r_{\sigma(k+1)}-1}{2}\rfloor.$$
If $n$ is odd, let $n=2t+1$, then
\begin{align*}
r_{\sigma(k)}+ \lfloor\frac{r_{\sigma(k+1)}-1}{2}\rfloor&\geq (2t+1)+1+  \lfloor\frac{(2t+1)+1-1}{2}\rfloor\\
&\geq 3t+2\\
&\geq \lfloor\frac{3(2t+1)}{2}\rfloor+1.
\end{align*}

If $n$ is even, let $n=2t$, then
\begin{align*}
r_{\sigma(k)}+ \lfloor\frac{r_{\sigma(k+1)}-1}{2}\rfloor&\geq (2t)+1+  \lfloor\frac{(2t)+1-1}{2}\rfloor\\
&\geq 3t+1\\
&\geq \lfloor\frac{3((2t))}{2}\rfloor+1.
\end{align*}
 \qed
\end{proof}

\begin{proof}[of Theorem~\ref{th:conj}]
From Lemma~\ref{lem:imply2}, Lemma~\ref{lem:case1}, Lemma~\ref{lem:case2}, Lemma~\ref{lem:case3} and Lemma~\ref{lem:case4},
Theorem~\ref{th:conj} holds if:
\begin{align*}
&r_{\sigma(k)} \geq n+2;\\
&r_{\sigma(k)} \leq n \; \text{and}\; k \geq 8;\\
&r_{\sigma(k)}=n+1\; \text{and}\; r_{\sigma(k+1)}\geq n+4;\\
&r_{\sigma(k)}=n+1,\; r_{\sigma(k+1)}\leq n+3\; \text{and}\; k \geq 10.\\
\end{align*}

Combining with Lemma~\ref{lem:pre}, we conclude.

\end{proof}


\bibliographystyle{splncs03}
\bibliography{biblio}

\begin{thebibliography}{10}
\providecommand{\url}[1]{\texttt{#1}}
\providecommand{\urlprefix}{URL }

\bibitem{berge}
Berge, C.: The Theory of Graphs and Its Applications. Greenwood Press (1982)

\bibitem{Brlekli}
Brlek, S., Li, S.: On the number of squares in a finite word (2022),
  \url{https://arxiv.org/abs/2204.10204}

\bibitem{DezaFT15}
Deza, A., Franek, F., Thierry, A.: How many double squares can a string
  contain? Discrete Applied Mathematics  180,  52--69 (2015)

\bibitem{FraenkelS98}
Fraenkel, A.S., Simpson, J.: How many squares can a string contain? J. Comb.
  Theory, Ser. {A}  82(1),  112--120 (1998)

\bibitem{Ilie}
Ilie, L.: A note on the number of distinct squares in a word. In: Brlek, S.,
  Reutenauer, C. (eds.) Proc.~Words2005, 5-th International Conference on
  Words. vol.~36, pp. 289--294. Publications du LaCIM, Montreal, Canada
  (13--17~Sep 2005)

\bibitem{Jono}
Jonoska, N., Manea, F., Seki, S.: A stronger square conjecture on binary words.
  In: Geffert, V., Preneel, B., Rovan, B., {\v{S}}tuller, J., Tjoa, A.M. (eds.)
  SOFSEM 2014: Theory and Practice of Computer Science. pp. 339--350. Springer
  International Publishing, Cham (2014)

\bibitem{KUBICA}
Kubica, M., Radoszewski, J., Rytter, W., Waleń, T.: On the maximum number of
  cubic subwords in a word. Eur. J. Comb.  34(1),  27--37 (2013)

\bibitem{lam}
Lam, N.H.: On the number of squares in a string. AdvOL-Report  2 (2013)

\bibitem{li2022}
Li, S.: On the number of $k$-powers in a finite word. Adv. Appl. Math.  139,
  102371 (2022)

\bibitem{li202202}
{Li}, S., {Pachocki}, J., {Radoszewski}, J.: {A note on the maximum number of
  $k$-powers in a finite word} (2022), \url{https://arxiv.org/abs/2205.10156}

\bibitem{thie}
Thierry, A.: A proof that a word of length n has less than 1.5n distinct
  squares (2020)

\end{thebibliography}
\end{document}